\documentclass[10pt,journal]{IEEEtran}
\usepackage{amsmath}           %
\usepackage{amssymb}
\usepackage{lipsum,adjustbox}                               %
\usepackage{pgfplots}
\usepackage{cite}
\usepackage[latin1]{inputenc}  %
\usepackage{graphicx}          %
\pgfplotsset{compat=1.16}
\usetikzlibrary{pgfplots.groupplots}

\usepackage{algorithm}
\usepackage{algpseudocode}
\usepackage{enumerate}
\usepackage{epstopdf}
\usepackage{multirow}

\usepackage{amsthm}
\newtheorem{thm}{Theorem}
\theoremstyle{definition}
\theoremstyle{remark}
\newtheorem{rem}{Remark}%
\theoremstyle{plain}
\newtheorem{lem}[thm]{Lemma}
\theoremstyle{remark}
\theoremstyle{definition}

\DeclareMathOperator{\E}{\mathbb{E}}
\DeclareMathOperator{\Cov}{Cov}

\DeclareMathOperator*{\argmin}{arg\,min}

\DeclareMathOperator{\tr}{tr}
\newcommand{\mbf}[1]{\mathbf{#1}}
\newcommand{\mbs}[1]{\boldsymbol{#1}}
\newcommand{\what}[1]{\widehat{#1}}

\newcommand{\real}{\mathbb{R}}

\newcommand{\T}{\top}
\newcommand{\0}{\mbf{0}}
\newcommand{\I}{\mbf{I}}

\newcommand{\eye}{\mbf{I}}
\newcommand{\out}{\mbf{y}}

\newcommand{\genp}{\mbf{x}}
\newcommand{\truelinp}{\mbs{\theta_\circ}}
\newcommand{\linp}{\mbs{\theta}}

\newcommand{\covp}{\mbf{C}}
\newcommand{\weight}{\mbf{W}}
\newcommand{\genmat}{\mbf{M}}
\newcommand{\covgenp}{\mbf{D}}
\newcommand{\eps}{\mbs{\varepsilon}}
\newcommand{\covnoise}{\mbf{V}}
\newcommand{\truecovnoise}{\mbf{V_\circ}}
\newcommand{\reg}{\mbs{\Phi}}
\newcommand{\linphat}{\what{\linp}}

\newcommand{\covout}{\mbf{R}}
\newcommand{\diag}{\text{diag}}

\newcommand{\linpset}{\Theta}

\newcommand{\covset}{\mathcal{S}}
\newcommand{\covpset}{\mathcal{C}}
\newcommand{\covgenpset}{\mathcal{D}}
\newcommand{\covnoiseset}{\mathcal{V}}
\newcommand{\range}{\mathcal{R}}

\newcommand{\covpu}{\mbf{U}}
\newcommand{\covpuc}{\mbf{u}}
\newcommand{\covpe}{\mbf{\Lambda}}

\newcommand{\mse}{\textsc{mse}}
\newcommand{\mspe}{\textsc{msd}}
\newcommand{\mad}{\textsc{mad}}
\newcommand{\covreg}{\mbs{\Sigma}}

\newcommand{\blue}{\textsc{Blue}}
\newcommand{\lmmse}{\textsc{Lmmse}}
\newcommand{\ls}{\textsc{Ls}}
\newcommand{\nmse}{\textsc{Nmse}}

\begin{document}

\title{Tuned Regularized Estimators for \\ Linear Regression via Covariance Fitting}
\author{Per Mattsson, Dave Zachariah, Petre Stoica\thanks{This
    work has been partly supported by the Swedish Research Council
    (VR) under contracts 2018-05040 and 2021-05022.}}

\maketitle

\begin{abstract}
We consider the problem of finding tuned regularized parameter estimators for
linear models. We start by showing that three known optimal linear
estimators belong to a wider class of estimators that can be
formulated as a solution to a weighted and constrained minimization
problem. The optimal weights, however, are
typically unknown in many applications. This begs the question, how
should we choose the weights using only the data? We propose using the
covariance fitting SPICE-methodology to obtain data-adaptive
weights and  show that the resulting class of estimators yields tuned versions of known regularized estimators -- such as ridge regression, LASSO, and regularized least absolute deviation. These theoretical results unify several important estimators under a common umbrella. The resulting tuned estimators are also shown to be practically relevant by means of a number of numerical examples.

\end{abstract}

\section{Introduction}

The linear model
\begin{equation}\label{eq:model}
\out = \reg \truelinp + \eps,
\end{equation}
has a wide range of applications in statistics, signal processing and
machine learning. Here $\out$ denotes a column vector consisting of $n$ samples,
$\reg$ is an $n \times d$ matrix of regressors, $\truelinp$ is an unknown
parameter vector and $\eps$ is a vector of zero-mean noise with covariance matrix $\Cov[\eps] = \truecovnoise \succeq \0$.

The least-squares (\ls) method is the standard approach to estimate
$\linp$. However, in applications with few samples $n$, high noise levels,
or heteroscedastic noise, it can suffer from large
errors. Regularized estimators,
such as Ridge regression \cite{hoerl1970ridge}, Lasso
\cite{tibshirani1996regression} and regularized LAD \cite{wang2006regularized}, alleviate these drawbacks but require separate
methods for tuning regularization parameters. In this paper, we are
interested in studying a class of tuned regularized estimation methods. 

We begin by showing how optimal linear estimators are derived in a class of
estimators that is parameterized by positive semi-definite weight
matrices. Subsequently, we show that setting these weight matrices
in a data-adaptive manner using a covariance-fitting criterion leads directly to
tuned versions of several popular estimators, such as square-root Ridge
estimators, square-root LASSO, and regularized LAD.

The covariance-fitting criterion is based on the SPICE-methodology, first
proposed in \cite{stoica2010new}, which we here generalize to include singular
covariance matrices. Our analysis considers a broader classes of
covariance structures than considered in
\cite{babu2014connection,stoica2014weighted,zachariah2015online} and
therefore it extends
and unifies results in the cited references.

\emph{Notation:} $\| \genp \|_2, \| \genp \|_1$ denote the $\ell_2$- and $\ell_1$-norms. $\| \mbf{X} \|_{\weight} = \sqrt{\tr\{ \mbf{X}^\top \weight \mbf{X} \}}$, where $\weight \succeq \0$, denote a weighted (semi)-norm. The Moore-Penrose pseudoinverse of $\mbf{X}$ is denoted $\mbf{X}^\dagger$. $\range(\genmat)$ is the range space of $\genmat$.
\section{Optimal linear estimators}\label{sec:optimal}

We being by considering linear estimators, i.e., estimators of the form
$\linphat = \mbf{K} \out$, where $\mbf{K}$ is a $d \times n$ matrix
that is independent of $\out$. Let the mean-squared error of $\linphat$
be denoted as 
\[
	\mse(\truelinp) = \E\left[ \| \truelinp - \linphat \|_2^2 \right],
\]
where the expectation is taken with respect to the noise $\eps$. We
will now show that three different optimal linear estimators belong to a
unified class of estimators.

\begin{thm}\label{thm:blue}
Consider the following class of estimators,
\begin{equation}\label{eq:blue_opt}
\begin{aligned}
	\linphat(\covnoise) = \argmin_{\linp} \quad& \| \out - \reg\linp \|_{\covnoise^\dagger}^2 \\
	\text{s.t.}\quad & \out - \reg \linp \in \range(\covnoise)
\end{aligned}
\end{equation}
spanned by all $\covnoise \succeq \0$.

If $\reg$ has full column rank, then among all linear \emph{unbiased}
estimators,  the minimum MSE is attained by $\linphat(\alpha
\covnoise_\circ)$ for any $\alpha > 0$ (see Appendix~\ref{app:blue} for an analytical expression). This is also known as the `best linear unbiased
estimator' (\blue) \cite{Rao1973_linear,SoderstromStoica1988_system}.
\end{thm}
\begin{proof} See Appendix~\ref{app:blue}.\end{proof}

\begin{rem}
	Without the constraint, there would
        be no penalty on parts of the residual $\out - \reg \linp$
        outside the range of $\covnoise$. In the case of $\covnoise =
        \alpha \covnoise_\circ$, the constraint is necessary in order
        to obtain the optimal unbiased estimator in general. Removing the constraint in
        \eqref{eq:blue_opt} yields the optimal estimator if and only if $\covnoise_\circ^\dagger \truecovnoise \reg = \reg$.
      \end{rem}
The optimal unbiased estimator can improve $\mse(\truelinp) $ over
\ls{} in problems with heteroscedastic or correlated noise. But
when $n$ is small, or $\reg$ is ill conditioned, this linear
estimator can still suffer from large errors. To cope with such cases we relax the
unbiasedness requirement and consider more general linear estimators that
minimize $\mse(\truelinp)$.
Specifically, we consider
a class of estimators formed by regularizing the criterion in
\eqref{eq:blue_opt} as follows:
\begin{equation}\label{eq:original_opt}
\linphat(\covp, \covnoise) = \argmin_{\linp \in \linpset(\covp, \covnoise)} \| \out - \reg \linp \|_{\covnoise^\dagger}^2 + \| \linp \|_{\covp^\dagger}^2
\end{equation}
where $\covp$ and $\covnoise$ are positive semi-definite weight
matrices. The parameter vector is restricted to the set
\[
	\linpset(\covp, \covnoise) = \left\{ \linp : \out - \reg \linp \in \range(\covnoise) \text{ and } \linp \in \range(\covp) \right\}.
\]
This constraint ensures that it is not possible to hide parts of the residuals or the parameter vector in a subspace that is not penalized when $\covnoise$ or $\covp$ are singular.

To study the feasibility of the minimization problem in \eqref{eq:original_opt}, we introduce the matrix
\[
	\covout(\covp, \covnoise) \triangleq \reg \covp \reg^\top +
        \covnoise  \; \succeq \0.
\]
When there is no risk of confusion we will drop the arguments and just write $\covout$.
\begin{thm}\label{thm:original_sol}
	If $\out \in \range(\covout)$, then a unique solution to
        \eqref{eq:original_opt} exists and is given by
	\[
          \linphat(\covp, \covnoise) \equiv \covp \reg^\top
                \covout^\dagger \out ,
              \]
        which is linear in $\out$.
	If $\out \notin \range(\covout)$, then $\linpset(\covp,
        \covnoise)$ is empty and thus \eqref{eq:original_opt} is infeasible.
\end{thm}
\begin{proof} See Appendix~\ref{app:original_sol}.\end{proof}

\begin{thm}\label{thm:optimal_cov}
	Among all linear estimators, the minimum MSE is attained by
        $\linphat(\covp, \covnoise)$ with weight matrices
	\[
		 \covp = \alpha \truelinp\truelinp^\top \quad
                 \text{and} \quad \covnoise = \alpha \covnoise_{\circ},
	\]
	for any $\alpha > 0$. 
\end{thm}
\begin{proof} See Appendix~\ref{app:optimal}. \end{proof}
\begin{rem}
It follows that the class of estimators \eqref{eq:original_opt}
includes the optimal linear estimator, but it is unrealizable since it
depends on $\truelinp$ and $\truecovnoise$, which are typically
unknown. Also, by setting $\covp = c \I$,  the estimator $\linphat(\covp,
\covnoise)$ also includes $\linphat(\covnoise)$ in \eqref{eq:blue_opt} when $c \rightarrow \infty$.
\end{rem}

Given the practical unrealizability of the optimal linear estimator, a common model-based approach is to consider a prior distribution over
$\truelinp$, with mean and covariance
\[
	\E[\truelinp] = \0, \quad \text{and} \quad \Cov[\truelinp] = \covp_\circ \succeq \0
      \]
Then $\E[ \mse(\truelinp)]$ will denote the mean squared error
\emph{marginalized} over all plausible $\truelinp$. The parameter
$\truelinp$ is here drawn independently from the measurement noise
$\eps$ in \eqref{eq:model}.
      
      \begin{rem}
        The zero-mean assumption does not incur any loss of generality, since
        any non-zero mean can be removed from the data $\out$.
      \end{rem}

\begin{thm}\label{thm:lmmse}
	Among all linear estimators, the minimum \emph{marginalized} MSE is attained by
        $\linphat(\covp, \covnoise)$ with weight matrices
	\[
		 \covp = \alpha \covp_\circ  \quad
                 \text{and} \quad \covnoise = \alpha \covnoise_{\circ},
	\]
	for any $\alpha > 0$. This is also known as the `linear minimum mean-square
estimator' (\lmmse{}) \cite{Kay1993_sspestimation,kailath2000linear}.

\end{thm}
\begin{proof}
See Appendix~\ref{app:optimal}.
\end{proof}

In summary, we see that $\linphat(\covp, \covnoise) $ encompasses
three different optimal linear estimators depending on the choice of
weight matrices $\covp$ and $\covnoise$. While the minimizer of the
MSE is unrealizable since it depends on $\truelinp$, the minimizer
of the marginal MSE instead defers the problem to the appropriate specification
of $\covp_\circ$.

In the following section, we will no longer restrict the discussion to  
the class of linear estimators and instead consider $\linphat(\covp, \covnoise)$  when the weight
matrices $\covp$ and $\covnoise$ depend on the data $\out$ and $\reg$.

\section{Data-dependent weight matrices}\label{sec:est_weights}
In the model-based approach, the weight matrices $\covp$ and
$\covnoise$ in \eqref{eq:original_opt} can be viewed as the covariance
matrices for $\linp$ and $\eps$, respectively. Consequently,
$\covout(\covp, \covnoise)$ is the (marginal) covariance matrix of $\out$. 

A possible way to fit $\covp$ and $\covnoise$ to the data is to use the criterion
\begin{equation}
	 \begin{aligned}
		 (\covp^*, \covnoise^*) =  \argmin_{(\covp, \covnoise )  \in \mathcal{S}}\quad &\left \| \out \out^\top  - \covout \right\|^2_{\covout^\dagger}\\
		 \text{s.t.} \quad & \out \in \range(\covout),
  \end{aligned}
  \label{eq:spice} 
\end{equation}
where the constraint $\out \in \range(\covout)$ ensures that the resulting $\covout^* = \covout(\covp^*, \covnoise^*)$ is indeed a valid covariance matrix for $\out$ even if $\covout^*$ is singular. 
The set
\begin{equation}
	\covset = \{ (\covp, \covnoise) : \covp \in \covpset,
        \covnoise \in \covnoiseset \}
 \label{eq:covstructure}
\end{equation}
determines the types of covariance matrices under consideration. We
assume that both $\covpset$ and $\covnoiseset$ include positive
definite matrices. This ensures that for any measurement $\out$ there exist $(\covp, \covnoise) \in \covset$ such that $\out \in \range(\covout)$.

The
fitting criterion \eqref{eq:spice} generalizes the criterion proposed in
\cite{stoica2010new} to handle potentially singular $\covout$ and we
will also extend the analysis in
\cite{babu2014connection,stoica2014weighted,zachariah2015online} to
consider the cases when the weight matrices are either
\begin{itemize}
\item of the form $\kappa \eye$ with $\kappa > 0$,
\item or diagonally structured positive semi-definite matrices,
\item or unstructured positive semi-definite matrices.
\end{itemize}

We will study the resulting estimator $\linphat(\covp^*,
\covnoise^*)$ in~\eqref{eq:original_opt}, using a fitted $(\covp^*,
\covnoise^*)$ from \eqref{eq:spice} . While Theorem~\ref{thm:original_sol} and the constraint in \eqref{eq:spice} ensures that
$\linphat(\covp^*, \covnoise^*)$ exist and is unique for any given $(\covp^*, \covnoise^*)$, 
 there may be multiple solutions to
\eqref{eq:spice} in general. Therefore we define the set of estimates:
\begin{equation}\label{eq:linpset}
\linpset^* \triangleq \left \{  \linphat(\covp^*, \covnoise^*) \: : \: (\covp^*,
  \covnoise^* ) \text{ in \eqref{eq:spice}} \right\}
\end{equation}
The main results of this paper are to characterize the solution set
$\linpset^*$ as tuned versions of several known regularized
estimators thus unifying them under the same umbrella.
\section{Main results}\label{sec:results}
In this section, we show that the estimators with data-dependent weight
matrices in \eqref{eq:linpset} correspond to several known tuned regularized estimators. The derivations are deferred to Section~\ref{sec:proof}. For notational simplicity, we let $\covreg = \frac{1}{n} \reg^\top
\reg \succeq \0$ denote the sample covariance matrix of the regressor
vectors and also let
\[
	\mspe(\linp) = \frac{1}{n} \| \out - \reg \linp\|_2^2 \quad
        \text{and} \quad \mad(\linp) = \frac{1}{n} \| \out - \reg \linp\|_1
\]
denote the (empirical) mean squared/absolute deviation. Recall that the least-squares and least
absolute deviation estimators are the minimizers of $\mspe(\linp)$ and
$\mad(\linp)$, respectively.

\subsection{Diagonally structured weight matrices}

In this section, we consider cases when both weight matrices $\covp$
and $\covnoise$ have diagonal structures. We will show that
$\linphat(\covp^*, \covnoise^*)$ in \eqref{eq:linpset} above is then a
minimizer of one of the following criteria:
\begin{align}
\sqrt{\mspe(\linp)} &+ \lambda \|
\linp\|_2 \quad & \text{(L2-L2)} \label{eq:l2-l2}
 \\
\mad(\linp)  &+ \lambda  \|
\linp\|_2 \quad & \text{(L1-L2)} \label{eq:l1-l2}
\\
\sqrt{\mspe(\linp)}  &+ \lambda  \|\sqrt{\eye
          \odot \covreg} \linp\|_1 \quad &
                                           \text{(L2-WL1)}  \label{eq:l2-wl1}\\ 
\mad(\linp)  &+ \lambda  \|\sqrt{\eye
	\odot \covreg} \linp\|_1 \quad & \text{(L1-WL1)}
\label{eq:l1-wl1}
\end{align}
with a \emph{tuned} regularization parameter $\lambda$.

We note that the minimizers of the above criteria correspond to the following popular regularized estimators:
square-root ridge regression for
\eqref{eq:l2-l2} \cite{gruber2017improving}; $\ell_2$-penalized least absolute deviation for
\eqref{eq:l1-l2}; square-root LASSO for \eqref{eq:l2-wl1} \cite{BelloniEtAl2011_squarerootlasso}; and $\ell_1$-penalized least absolute deviation for
\eqref{eq:l1-l2} \cite{wang2013l1}. A few properties of these estimators are worth mentioning:
\begin{enumerate}
\item  Criteria using $\mad(\linp)$ (i.e., \eqref{eq:l1-l2} and
\eqref{eq:l1-wl1}) are known to be better suited for
problems with noise outliers than those based on $\mspe(\linp)$
(i.e., \eqref{eq:l2-l2} and
\eqref{eq:l2-wl1}).
\item The weighted regularization term in
\eqref{eq:l2-wl1} and \eqref{eq:l1-wl1} corresponds to standardizing
the regressor variables by their (empirical) standard
deviations. This $\ell_1$-regularization term is suited for problems with sparse parameter vectors.
\item The minimizer for a criterion containing the square-root fitting term
  $\sqrt{\mspe(\linp)} $ is also the minimizer of a criterion with
  $\mspe(\linp)$, but with a different $\lambda$. That is, for \eqref{eq:l2-l2} there is a corresponding ridge regression criterion, and for \eqref{eq:l2-wl1} there is a corresponding LASSO-criterion.
 \end{enumerate}

Our main result is that $\linphat(\covp^*, \covnoise^*)$ in
\eqref{eq:linpset} yields tuned regularized estimators according to
Table~\ref{tab:estimators}: The structure of $\covnoise$ determines
the data-fitting term of the criterion, while the structure of $\covp$ determines the
regularization term and the parameter $\lambda$. Choosing between 
nonuniform and uniform diagonal structures of $\covnoise$ thus depends on whether
the measurement $\out$ is subject to noise outliers or not, while
choosing between nonuniform and uniform diagonal structures of
$\covp$ depends on whether the unknown $\truelinp$  is sparse or
not.  This result unifies and extends the connections between the covariance
fitting and regularized estimation developed in
\cite{babu2014connection,stoica2014weighted,zachariah2015online}.

\begin{table}[t]
  \centering
  \caption{Estimator $\linphat(\covp^*, \covnoise^*)$ with diagonally
    structured weight matrices  in
    \eqref{eq:linpset} 
    minimizes four different criteria.}
\begin{tabular}{c|cc}
                        & $\covp = c \I$ & $\covp = \diag(\mathbf{c})$ \\ \hline
  $\covnoise = c \I$ & \eqref{eq:l2-l2}, $\lambda = \sqrt{\frac{\tr\{ \covreg \}
                       }{n}}$           & \eqref{eq:l2-wl1}, $\lambda = \frac{1}{\sqrt{n}}$ \\
  $\covnoise = \diag(\mathbf{v})$ & \eqref{eq:l1-l2}, $\lambda = \sqrt{\frac{\tr\{ \covreg \} }{n}}$
                                         & \eqref{eq:l1-wl1}, $\lambda = \frac{1}{\sqrt{n}}$ \\
\end{tabular}
\label{tab:estimators}
\end{table}

\subsection{Unstructured weight matrices}

The derivations of the results above, to be presented in
Section~\ref{sec:proof},  also cover the case of unstructured weight matrices.

The cases with unstructured $\covnoise$ can readily be dismissed as
uninteresting:  When this matrix can be any positive semi-definite
matrix, then the output in \eqref{eq:spice} can be explained
completely by the noise, setting $\covnoise^* = \out \out^\top$ and
$\covp^* = \0$. This gives $\linphat(\covp^*, \covnoise^*) = \0$,
an uninteresting estimator.

Let us therefore consider cases when $\covp$ is unstructured:
\begin{itemize}
	\item If $\covnoise = v\eye$, then 
		\begin{align*}
			\linphat(\covp^*, \covnoise^*) &\in \argmin_{\linp} \sqrt{\mspe(\linp)} + \frac{1}{\sqrt{n}} \| \linp\|_{\covreg} \\
						       &= (1-q) \argmin_{\linp} \mspe(\linp)
		\end{align*}
		where 
\[
	q = \begin{cases} \frac{1}{\sqrt{n-1}} \frac{ \| \out \|_{\eye - \reg \reg^\dagger}}{\| \out \|_{\reg\reg^\dagger}} & \text{if } \frac{1}{\sqrt{n-1}} \frac{\|\out\|_{\eye-\reg\reg^\dagger}}{\|\out\|_{\reg\reg^\dagger}} < 1 \\ 1 & \text{otherwise} \end{cases}
\]
	\item If $\covnoise = \diag(\mbf{v})$, then
		\[
		\linphat(\covp^*, \covnoise^*) \in \argmin_{\linp} \mad(\linp) + \frac{1}{\sqrt{n}} \| \linp\|_{\covreg}.
		\]
\end{itemize}
It can be noted that setting $\covnoise = v\eye$ yields 
the posterior mean of $\truelinp$ using a g-prior and a Gaussian data model
\cite{zellner1986assessing}. The parameter $q$ can intuitively be seen as
an estimate of an inverse signal-to-noise ratio, so the criterion shrinks the least squares solution towards zero if the estimated signal-to-noise ratio is low. 

While these cases are of theoretical interest, their practical
relevance is limited since the resulting estimators usually are not sufficiently regularized.
For example, if there
exist $\linp$ such that $\out = \reg \linp$, then all
$\linphat(\covp^*, \covnoise^*)$ will be ordinary least squares
solutions, as shown in Appendix~\ref{app:general}. Since this
typically happens when $n<d$, the method offers no regularization in
this important scenario.

For these reasons we believe that unstructured weight matrices have less practical importance.

\section{Numerical experiments}\label{sec:num}
In this section we will evaluate the tuned methods in three
different settings where regularization can improve over the standard
\ls{} method. In each setting we use a fixed $\reg$ with elements drawn from an i.i.d. zero mean
Gaussian distribution, and generate $\out$ as
\[
	\out = \reg \truelinp + \eps, \quad \truelinp \sim \mathcal{N}(\0, \covp_\circ), \quad \eps \sim \mathcal{N}(\0, \covnoise_\circ), 
\]
or equivalently $\out \sim \mathcal{N}(\0, \covout(\covp_\circ,
\covnoise_\circ))$. The following three cases will be considered:
\begin{enumerate}
	\item $\covnoise_\circ = v \eye$ and $\covp_\circ = \eye$.
	\item $\covnoise_\circ = v \eye$ and $\covp_\circ$ is a diagonal matrix with only 10 non-zero elements. This means that $\truelinp$ is sparse, with only 10 non-zero elements.
	\item $\covp_\circ$ is diagonal with only 10 non-zero elements. $\covnoise_{\circ}$ is first set equal to $v \eye$. Then two elements are changed to 500. This means that $\truelinp$ is sparse, and there are two outliers in the data.
        \end{enumerate}
In all cases $n=d=100$ and $v$ is chosen so that the signal-to-noise ratio is
\begin{equation*}
\text{SNR} = \frac{\tr\{ \reg \covp_\circ \reg^\T \} }{\tr\{
  \covnoise_\circ \}} = 10
\end{equation*}
Now consider an estimator $\linphat_{\lambda}$ which
minimizes any given regularized criterion
\eqref{eq:l2-l2}-\eqref{eq:l1-wl1} with a parameter $\lambda$. 
We evaluate its performance using the marginalized and normalized mean square error,
\[
	\nmse(\lambda) = \frac{\E[ \| \truelinp -
            \linphat_\lambda \|_2^2 ]}{\tr(\covp_\circ)},
      \]
that is approximated using 1000 Monte-Carlo simulations. Note that
$\nmse(0)$ is the performance of an unregularized estimator and as
$\lambda \rightarrow \infty$ we have that
$\nmse(\lambda) \rightarrow 1$ since $\linphat_\lambda \rightarrow
\0$. We also show the \nmse{} of the oracle estimator
$\linphat(\covp_\circ,  \covnoise_\circ)$, which is a lower bound on the error.

Figure~\ref{fig:numericalnmse} displays $\nmse(\lambda)$ as a
function of $\lambda$ using the four regularized estimators
in the three cases above. In each case we show the lower bound as well
as the \emph{tuned} $\lambda$ that follows from using the fitted
weight matrices (see Table~\ref{tab:estimators}). Note that in all
cases, regularization can reduce the error below $\nmse(0)$.

\begin{figure*}[ht!]
	\centering
	\begin{adjustbox}{width=\textwidth}
		\begin{tikzpicture}
	\begin{groupplot}[
		group style={
			group size=4 by 3,
			vertical sep=.5cm
		},
		width=0.9\columnwidth,
		height=5cm,
		xmin =0,
		xmax=1.5,
		ymin=0,
		ymax=1.5,
		restrict y to domain = 0:7000
		]
		\nextgroupplot[title={L2-L2}, ylabel={Case 1, NMSE}, xticklabel={\empty}]
			\addplot+[no marks, thick] plot table[col sep=comma] {figures/l2l2full.csv};
			\addplot[no marks, thick, orange] coordinates {(0, 0.263) (1.7, 0.263)};
			\addplot[mark=*, red] coordinates {(1.0), (0.432)};

		\nextgroupplot[title={L1-L2}, yticklabels={\empty}, xticklabel={\empty}]
			\addplot+[no marks, thick] plot table[col sep=comma] {figures/l1l2full.csv};
			\addplot[no marks, thick, orange] coordinates {(0, 0.276) (1.7, 0.263)};
			\addplot[mark=*, red] coordinates {(1.0), (0.644)};

		\nextgroupplot[title={L2-WL1}, xmax=0.8, yticklabels={\empty}, xticklabel={\empty}]
			\addplot+[no marks, thick] plot table[col sep=comma] {figures/l2l1full.csv};
			\addplot[no marks, thick, orange] coordinates {(0, 0.276) (1.7, 0.263)};
			\addplot[mark=*, red] coordinates {(0.1), (0.4199)};

		\nextgroupplot[title={L1-WL1}, xmax=0.8, yticklabels={\empty}, xticklabel={\empty}]
			\addplot+[no marks, thick] plot table[col sep=comma] {figures/l1l1full.csv};
			\addplot[no marks, thick, orange] coordinates {(0, 0.276) (1.7, 0.263)};
			\addplot[mark=*, red] coordinates {(0.1), (0.581)};

		\nextgroupplot[ylabel={Case 2, NMSE}, xticklabel={\empty}]
			\addplot+[no marks, thick] plot table[col sep=comma] {figures/l2l2sparse.csv};
			\addplot[no marks, thick, orange] coordinates {(0, 0.011) (1.7, 0.011)};
			\addplot[mark=*, red] coordinates {(1.0), (0.448)};

		\nextgroupplot[yticklabels={\empty}, xticklabel={\empty}]
			\addplot+[no marks, thick] plot table[col sep=comma] {figures/l1l2sparse.csv};
			\addplot[no marks, thick, orange] coordinates {(0, 0.011) (1.7, 0.011)};
			\addplot[mark=*, red] coordinates {(1.0), (0.668)};

		\nextgroupplot[yticklabels={\empty}, xmax=0.8, xticklabel={\empty}]
			\addplot+[no marks, thick] plot table[col sep=comma] {figures/l2l1sparse.csv};
			\addplot[no marks, thick, orange] coordinates {(0, 0.011) (1.7, 0.011)};
			\addplot[mark=*, red] coordinates {(0.1), (0.0446)};

		\nextgroupplot[yticklabels={\empty}, xmax=0.8, xticklabel={\empty}]
			\addplot+[no marks, thick] plot table[col sep=comma] {figures/l1l1sparse.csv};		
			\addplot[no marks, thick, orange] coordinates {(0, 0.011) (1.7, 0.011)};
			\addplot[mark=*, red] coordinates {(0.1), (0.0651)};

		\nextgroupplot[ylabel={Case 3, NMSE}, xlabel={$\lambda$}]
			\addplot+[no marks, thick] plot table[col sep=comma] {figures/l2l2sparse_outlier.csv};
			\addplot[no marks, thick, orange] coordinates {(0, 0.012) (1.7, 0.012)};
			\addplot[mark=*, red] coordinates {(1.0), (0.657)};

		\nextgroupplot[yticklabels={\empty}, xlabel={$\lambda$}]
			\addplot+[no marks, thick] plot table[col sep=comma] {figures/l1l2sparse_outlier.csv};
			\addplot[no marks, thick, orange] coordinates {(0, 0.012) (1.7, 0.012)};
			\addplot[mark=*, red] coordinates {(1.0), (0.659)};

		\nextgroupplot[yticklabels={\empty}, xmax=0.8, xlabel={$\lambda$}]
			\addplot+[no marks, thick] plot table[col sep=comma] {figures/l2l1sparse_outlier.csv};
			\addplot[no marks, thick, orange] coordinates {(0, 0.012) (1.7, 0.012)};
			\addplot[mark=*, red] coordinates {(0.1), (0.3634)};

		\nextgroupplot[yticklabels={\empty}, xmax=0.8, xlabel={$\lambda$}]
			\addplot+[no marks, thick] plot table[col sep=comma] {figures/l1l1sparse_outlier.csv};		
			\addplot[no marks, thick, orange] coordinates {(0, 0.012) (1.7, 0.012)};
			\addplot[mark=*, red] coordinates {(0.1), (0.065)};

	\end{groupplot}
	\end{tikzpicture}
\end{adjustbox}
\caption{The normalized mean-squared error $\nmse(\lambda)$ as a
  function of the regularization parameter $\lambda$ (blue curves) in
  three different cases. The
  horizontal lines indicate the lower bound set by an oracle estimator. We consider four
different estimators \eqref{eq:l2-l2}-\eqref{eq:l1-wl1} that are
obtained using the data-adaptive weight matrices in
\eqref{eq:original_opt}, see Table~\ref{tab:estimators} and the
corresponding red dots on the curves.}
\label{fig:numericalnmse}
\end{figure*}
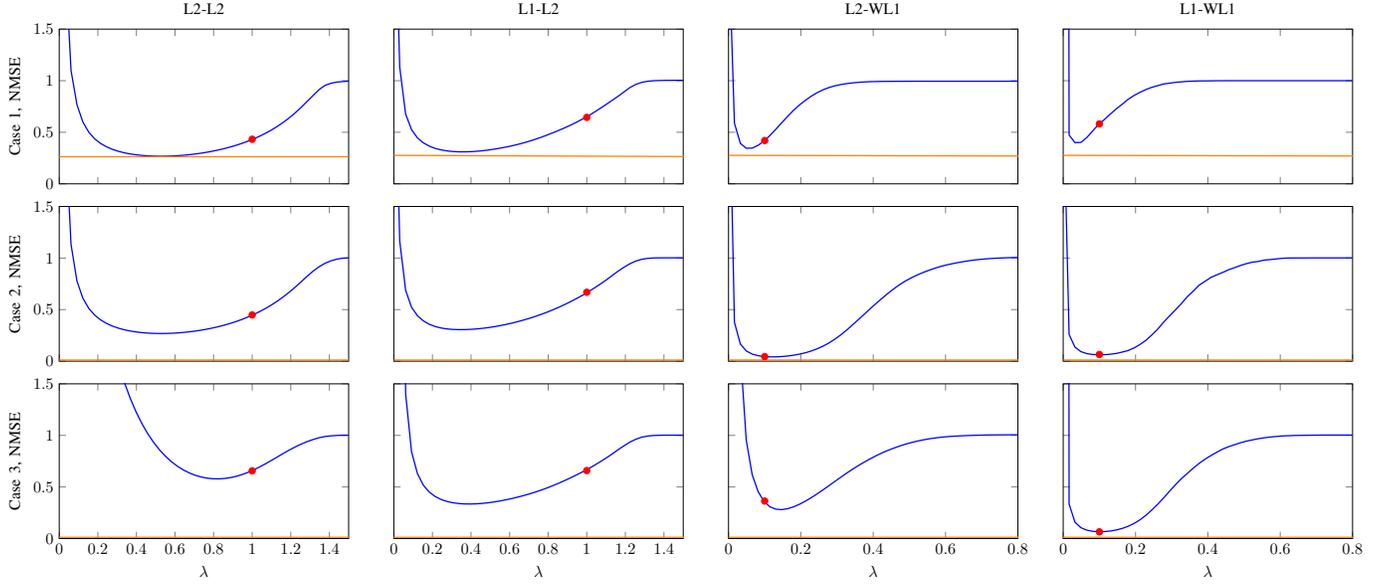

Case 1) with uniform noise power and a dense parameter vector:  We see
that the lower bound can be attained by an L2-L2 estimator (as
expected). We also see that using a nonuniform diagonal matrix
$\covnoise$ leads to slightly worse performance, while assuming a nonuniform $\covp$ does not hurt the performance in a noticeable way.

Case 2) with uniform noise power and a sparse parameter vector: Here
we see that using a nonuniform diagonal matrix
$\covp$ clearly outperforms the alternative, and the tuned L2-WL1 is close to
the lower bound. Again, using a nonuniform diagonal $\covnoise$ gives slightly worse performance.

Case 3) with nonuniform noise power and a sparse parameter vector:
When a nonuniform diagonal matrix $\covnoise$ is used, the performance is about the same as in Case 2 with no outliers. However, when a uniform diagonal matrix $\covnoise$ is used, the outliers impair the performance of the resulting estimators. 

In summary, when we let $\covpset$ and $\covnoiseset$ in
\eqref{eq:spice} have the same structure as the true covariance
matrices of the data generating process, then the corresponding
regularized minimization problems in
\eqref{eq:l2-l2}-\eqref{eq:l1-wl1} can be tuned to yield estimators
whose performance is quite close to the optimal oracle estimator. In
such cases, furthermore, we observe that corresponding tuned versions in
Table~\ref{tab:estimators} are close to the optimal tuning. Finally,
assuming a more general structure for $\covpset$ and $\covnoiseset$
than necessary does not hurt the performance much. These observations
suggest using a nonuniform diagonal structures for both $\covpset$ and
$\covnoiseset$ -- i.e., the L1-WL1 estimator -- if there is no prior knowledge about the data generating process.
\section{Derivation}\label{sec:proof} 
In this section we will show that the tuned regularized estimators presented in Section~\ref{sec:results} indeed give the estimates $\linphat(\covp^*, \covnoise^*)$ in \eqref{eq:linpset}.

Define the cost function \begin{equation}\label{eq:J} \begin{aligned}
	J(\linp; \covp, \covnoise) &\triangleq \|\out - \reg \linp\|_{\covnoise^\dagger}^2 + \| \linp\|_{\covp^\dagger}^2 + \frac{1}{\|\out\|_2^2} \tr\left\{ \covout\right\} \\
				   &= f(\out-\reg\linp, \covnoise, \eye) + f(\linp, \covp, \reg^\top \reg)
\end{aligned}
\end{equation}
where
\[
	f(\genp, \covgenp, \weight) \triangleq \| \genp \|_{\covgenp^\dagger}^2 + \frac{1}{\|\out\|_2^2} \tr\left\{ \weight \covgenp\right\}.
\]
Note that this is the criterion in \eqref{eq:original_opt} with a weighted trace of $\covout$ added. To see the connection between $J$ and the criterion in \eqref{eq:spice}, let
\begin{equation}
	F(\covp, \covnoise) \triangleq \min_{\linp \in \linpset(\covp, \covnoise)} J(\linp; \covp, \covnoise).
\end{equation}
\begin{thm}\label{thm:f}
	For $(\covp, \covnoise) \in \covset$ such that $\out \in \range(\covout(\covp, \covnoise))$, 
	\[
		F(\covp, \covnoise) = J(\linphat(\covp, \covnoise); \covp, \covnoise) = \frac{1}{\|\out\|_2^2} \| \out \out^\top - \covout \|_{\covout^\dagger}^2 +2,
	\]
	where $\linphat(\covp, \covnoise) = \covp \reg^\top \covout^\dagger \out$.
\end{thm}	
\begin{proof}
	See Appendix~\ref{app:f}.
\end{proof}
From this it can be seen that the minimization of $F(\covp, \covnoise)$ is equivalent to minimization of \eqref{eq:spice}. We will now switch the order of minimization. Let
\begin{equation}\label{eq:G}
	\begin{aligned}	G(\linp) \triangleq \inf_{(\covp, \covnoise) \in \covset} &\quad J(\linp; \covp, \covnoise) \\ \text{s.t.} &\quad \linp \in \linpset(\covp, \covnoise) \end{aligned}
\end{equation}
It will be seen below that in most cases the infimum in \eqref{eq:G} will be attained by some $(\covp, \covnoise)$ for all $\linp$. However, when $\covpset$ contains all positive semi-definite matrices and $\reg^\top \reg$ is singular, there is a special (but uninteresting) case where this is not true. So for generality we use $\inf$ instead of $\min$ here. The following theorem shows that as long as all $\linp_G \in \argmin_{\linp} G(\linp)$ are such that the infimum is attained, then
\[
	\linpset^* = \argmin_{\linp \in \real^d} G(\linp).
\]
\begin{thm}\label{thm:main}
The set $\linpset^*$ in \eqref{eq:linpset} satisfy
\[
	\linpset^* \subseteq \argmin_{\linp \in \real^d} G(\linp).
\]
Furthermore, for all
\[
	\linp_G \in \argmin_{\linp \in \real^d} G(\linp)
\]
for which the infimum in \eqref{eq:G} is attained, $\linp_G \in \linpset^*$.
\end{thm}
\begin{proof}
	See Appendix~\ref{app:main}
\end{proof}
Finally, we will explore how the choice of $\covpset$ and $\covnoiseset$ determines $G(\linp)$. Note that
\[
	G(\linp) = h(\out-\reg\linp, \eye; \covnoiseset) + h(\linp, \reg^\top\reg; \covpset)
\]
where
\begin{equation}\label{eq:h}
	\begin{aligned}
	h(\genp, \weight; \covgenpset) = \inf_{\covgenp \in \covgenpset} &\, f(\genp, \covgenp; \weight) \\
	\text{s.t. } &\, \genp \in \range(\covgenp)
\end{aligned}
\end{equation}
The following lemmas explore how different choices of the set $\covgenpset$ will affect the functional form of $h$.
\begin{lem}
	If $\covgenpset = \{ \covgenp : \covgenp = \kappa \eye \text{ with }  0 \leq \kappa < \infty \}$, then 
	\[
		h(\genp; \weight, \covgenpset) = \frac{2}{\|\out\|_2} \| \genp \|_2 \sqrt{\tr\{\weight\}},
	\]
	and the minimum in \eqref{eq:h} is attained by 
	\[
		\what{\covgenp} = \frac{\|\out\|_2 \| \genp\|_2}{\sqrt{\tr\{ \weight\}}} \eye.
	\]
\end{lem}
\begin{proof} For $\genp \neq \0$, just insert $\covgenp = \kappa \eye$ in \eqref{eq:h}, and set the derivative with respect to $\kappa$ to zero. If $\genp = \0$, then $\kappa = 0$ gives $h = 0$, which clearly is the minimum. \end{proof}
\begin{lem}\label{lem:diagonal}
	If $\covgenpset = \{ \covgenp : \covgenp = \text{\emph{diag}}(a_1, \ldots, a_d), 0 \leq a_i < \infty \}$, then 
	\[
		h(\genp; \weight, \covgenpset) = \frac{2}{\| \out\|_2} \left\| \sqrt{ \eye \odot \weight} \genp \right\|_1.
	\]
	and the minimum in \eqref{eq:h} is attained by
	\[
		\what{\covgenp} = \|\out\|_2 \text{\emph{diag}}\left( \frac{ | x_1|}{\sqrt{w_{1,1}}}, \ldots, \frac{|x_d|}{\sqrt{w_{d,d}}} \right).
	\]
	where $w_{i,i}$ is the $i$th diagonal elements of $\weight$.
\end{lem}
\begin{proof}
	See Appendix~\ref{app:diagonal}
\end{proof}
\begin{lem}\label{lem:full}
	If $\covgenpset = \{ \covgenp : \covgenp \succeq \0 \}$ then 
	\[
		h(\genp; \weight, \covgenpset) = \frac{2}{\| \out\|_2} \| \genp \|_{\weight}.
	\]
	The minimum in \eqref{eq:h} is attained if $\genp = \0$ or $\weight \genp \neq \0$, and is then given by
	\[
		\what{\covgenp} = \begin{cases} \frac{ \|\out\|_2}{\|\genp\|_{\weight}} \genp \genp^\top & \text{if } \weight \genp \neq \0 \\ \0 & \text{if } \genp = \0 \end{cases}.
	\]
\end{lem}
\begin{rem}\label{rem:full}
	In the case that $\genp \neq \0$ but $\weight \genp = \0$, the infimum in \eqref{eq:h} is zero but it is not attained by any finite $\covgenp$. 
	This can only occur in the term related to $\covp$, if $\weight = \reg^\top \reg$ is singular. It will only be a problem if $G(\linp)$ is minimized by $\linphat \neq \0$ such that $\reg \linphat = \0$. In this case $\linphat \notin \linpset^*$. However, then $G(\linp)$ is also minimized by $\linp = \0 \in \linpset^*$.
\end{rem}

With these lemmas we can take different combinations of $\covpset$ and $\covnoiseset$ and find the corresponding $G(\linp)$ to see that minimization of $G(\linp)$ is equivalent to the results in Section~\ref{sec:results}.

\section{Conclusion}
We began by showing that a weighted and constrained minimization problem spans a class of estimators that encompass three known optimal linear estimators. The constrained form ensures that singular covariance matrices can be handled, while the weight matrices determine the resulting estimator.

However, the optimal weight matrices depend on the unknown parameters, or their prior covariance matrix, as well as the noise covariance. Since these properties are typically unknown, how should the weight matrices be chosen based only on the data? The proposed method in this paper was to use the covariance-fitting SPICE-methodology to find data-adaptive weight matrices. Interestingly, while the class of estimators is an $\ell_2$-regularized form of weighted least-squares, using the data-adaptive weights yielded several different known tuned regularized estimators -- ridge regression, LASSO, and regularized least absolute deviation -- depending on the assumed structure of the unknown covariances matrices.  In this way the paper connects several important estimators, and also extends the analysis of the SPICE-methodology to singular covariance matrices. 

Finally a numerical experiment was performed. It was seen that when the covariance matrices of the data-generating process corresponded to the structure assumed in the SPICE-criterion, the resulting estimator is not far from the optimal one. Furthermore, assuming a more general structure than necessary does not incur any significant loss to performance. These observations suggest that it is sensible to assume nonuniform diagonal structure for the covariance matrices when no prior knowledge about the data-generating process is available, and thus use an L1-WL1 estimator.

\appendix

\subsection{Proof of Theorem~\ref{thm:blue}}\label{app:blue}
It can be seen that the optimization is feasible if and only if $\out \in \range(\reg\reg^\top + \covnoise)$, cf. Appendix~\ref{app:original_sol}. Also note that this is satisfied for any $\out$ generated according to \eqref{eq:model} if $\covnoise = \alpha\truecovnoise$.

The constraint $\out - \reg \linp \in \range(\covnoise)$ can be written as 
\[
	\covnoise^\dagger \covnoise (\out - \reg \linp) = \out - \reg \linp 
\] 
and with some slight rearrangement we get
\[
(\eye - \covnoise^\dagger \covnoise) \reg \linp = (\eye - \covnoise^\dagger \covnoise) \out.
\]
Hence $\linphat$ is an optimal solution if $\out - \reg\linphat \in \range(\covnoise)$ and there exist $\mbs{\lambda}$ such that
\begin{equation}\label{eq:optcond_blue}
	-\reg^\top \covnoise^\dagger( \out - \reg \linphat ) + \reg^\top( \eye - \covnoise^\dagger \covnoise) \mbs{\lambda} = \0.
\end{equation}
This solution is unique if $\reg$ has full column rank, since the constrained problem is strictly convex in this case.

In order to find a solution, we assume that $\out \in \range(\reg\reg^\top + \covnoise)$ so the problem is feasible. Hence we can write 
\[
	\out = (\reg\reg^\top + \covnoise) \genp
\]
for some $\genp$. 
We will now show that the optimality conditions are satisfied by
\begin{equation}\label{eq:blue_analytic}
	\linphat(\covnoise) = \reg^\dagger [ \eye - \covnoise \genmat (\genmat \covnoise \genmat)^\dagger\genmat] \out,
\end{equation}
where $\genmat = \eye - \reg \reg^\dagger$. Note that $\genmat \reg = \0$ and $\reg^\top \genmat = \0$. Hence $\genmat \out = \genmat \covnoise \genp$, and
\[
	\genmat \covnoise \genmat (\genmat \covnoise \genmat)^\dagger \genmat \out = \genmat\covnoise\genmat(\genmat\covnoise\genmat)^\dagger \genmat \covnoise \genp = \genmat \covnoise \genp,
\]
where the last equality can be seen by setting $\covnoise = \mbf{L} \mbf{L}^\top$ and then make use of the pseudo-inverse identity $ \mbf{X} \mbf{X}^\top (\mbf{X} \mbf{X}^\top)^\dagger \mbf{X} = \mbf{X}$. This can be used to see that
\begin{align*}
	\out - \reg\linphat &= \out - (\eye - \genmat)(\eye - \covnoise \genmat(\genmat \covnoise \genmat)^\dagger \genmat) \out \\
			 &= \covnoise \genmat(\genmat \covnoise\genmat)^\dagger \genmat \out \in \range(\covnoise).
\end{align*}
Hence, setting 
\[
	\mbs{\lambda} = - \genmat(\genmat \covnoise \genmat)^\dagger \genmat \out
\]
in \eqref{eq:optcond_blue} shows that \eqref{eq:blue_analytic} is indeed the optimal solution.

The theorem then follows by noting that \eqref{eq:blue_analytic} gives the \blue{} when $\reg$ has full rank, see e.g. \cite{searle1989blue}.

\subsection{Proof of Theorem~\ref{thm:original_sol}\label{app:original_sol}}
We first show that $\linpset(\covp, \covnoise)$ is non-empty if and only if $\out \in \range(\covout)$. First assume that $\linpset(\covp, \covnoise)$ is non-empty and that $\linp \in \linpset(\covp, \covnoise)$. Hence $\out - \reg \linp \in \range(\covnoise)$ so $\out - \reg \linp =\covnoise \genp$ for some $\genp$, and
\[
\out = \reg \linp + \covnoise \genp.
\]
Clearly $\covnoise \genp \in \range(\covnoise) \subseteq \range(\covout)$. Furthermore $\linp \in \range(\covp) = \range(\covp^{1/2})$, so
\[
	\reg \linp \in \range(\reg \covp^{1/2}) = \range(\reg \covp \reg^\top) \subseteq \range(\covout).
\]
With this we can conclude that $\out \in \range(\covout)$.

In the other direction, assume that $\out \in \range(\covout)$. We can clearly see that $\linphat = \covp \reg^\top \covout^\dagger \out \in \range(\covp)$. Furthermore $\out \in \range(\covout)$ implies that $\out = \covout \covout^\dagger \out$, so 
\begin{equation}\label{eq:res_in_V}
	\out - \reg \linphat = \covout\covout^\dagger \out - \reg \covp \reg^\top \covout^\dagger \out = \covnoise \covout^\dagger \out \in \range(\covnoise).
\end{equation}
Hence it can be concluded that $\linphat \in \linpset(\covp, \covnoise)$.

To show that there is a unique solution when $\out \in \range(\covout)$, note that $\| \linp \|_{\covp^\dagger}^2$ is strictly convex on $\linpset(\covp, \covnoise)$. Hence, the full problem is strictly convex, so it has a unique solution if it is feasible. 

The constraint $\linp \in \linpset(\covp, \covnoise)$ can be written as
\[
	\begin{bmatrix} (\eye - \covnoise^\dagger \covnoise) \reg \\ \eye - \covp^\dagger \covp \end{bmatrix} \linp = \begin{bmatrix} (\eye - \covnoise^\dagger \covnoise) \out \\ \0 \end{bmatrix}.
\]
Hence $\linphat$ is optimal if $\linphat \in \linpset(\covp, \covnoise)$ and there exists $\mbs{\lambda}$ such that
\[
	-\reg^\top \covnoise^\dagger (\out - \reg \linphat) + \covp^\dagger \linphat + \begin{bmatrix} \reg^\top (\eye - \covnoise^\dagger \covnoise) & \eye - \covp^\dagger \covp \end{bmatrix}\mbs{\lambda} = \0.
\]
Above we have seen that $\linphat = \covp \reg^\top \covout^\dagger \out \in \linpset(\covp, \covnoise)$ if $\out \in \range(\covout)$. 
Furthermore, by using \eqref{eq:res_in_V}, it can be seen that the optimality equation is satisfied with 
\[
	\mbs{\lambda} = \begin{bmatrix} -\covout^\dagger\out \\ \reg^\top \covout^\dagger\out  \end{bmatrix},
\]
so $\linphat = \covp \reg^\top \covout^\dagger \out$ is indeed the unique optimal solution if $\out \in \range(\covout)$.

\subsection{Proof of Theorem~\ref{thm:optimal_cov} and Theorem~\ref{thm:lmmse}}\label{app:optimal}
We here prove the theorems for the case that $\alpha = 1$, but note that scaling $\covp$ and $\covnoise$ with the same constant will not change $\linphat$.
Consider any linear estimator
\[
\linphat = \genmat \out.
\]
The MSE is then given by 
\[
	\mse(\truelinp) = \tr\left( (\eye - \genmat \reg) \truelinp \linp_\circ^\top (\eye - \genmat \reg)^\top + \genmat \covnoise_\circ \genmat \right).  
\]
To show Theorem~\ref{thm:optimal_cov} we set $\covp = \truelinp\linp_\circ^\top$ and $\covnoise = \covnoise_\circ$. In Theorem~\ref{thm:lmmse} we take the expectation over the MSE and thus instead use $\covp = \covp_\circ$ where $\covp_\circ = \E[ \truelinp \linp_\circ^\top]$. It follows that in both cases we want to find the $\genmat$ that minimize the trace of 
\begin{multline*}
	\mbf{X}(\genmat) =	(\eye - \genmat\reg) \covp (\eye-\genmat \reg)^\top + \genmat \covnoise \genmat^\top = \\
	\covp + \genmat \covout \genmat^\top - \genmat \reg \covp - \covp \reg^\top \genmat^\top,
\end{multline*}
where $\covout = \reg \covp \reg^\top + \covnoise$.
From Theorem~\ref{thm:original_sol} we know that minimizing \eqref{eq:original_opt} corresponds to 
\[
\genmat^* = \covp \reg^\top \covout^\dagger.
\]
Using the identity $\covout^\dagger \covout \covout^\dagger = \covout^\dagger$ it follows that
\[
	\mbf{X}(\genmat^*) = \covp - \covp \reg^\top \covout^\dagger \reg \covp.
\]
Using this we can see that for any $\genmat$
\[
	\mbf{X}(\genmat) - \mbf{X}(\genmat^*) = (\genmat - \covp \reg^\top \covout^\dagger) \covout ( \genmat - \covp \reg^\top \covout^\dagger)^\top \succeq \0.
\]
To see that this equality holds, just expand the right-hand side and use the fact that $\covout^\dagger \covout \reg \covp = \covout \covout^\dagger \reg\covp = \reg \covp$ since $\range(\reg \covp) \subseteq \range(\covout)$.
This shows that $\tr(\mbf{X}(\genmat)) \geq \tr(\mbf{X}(\genmat^*))$ so the two theorems follow.

\subsection{General $\covp$}\label{app:general}
In this section we will show that \eqref{eq:spice} with general positive semi-definite $\covp$ do not result in any regularization compared to least squares if there exist a $\linp$ such that $\out = \reg \linp$. In this case any least squares solution $\linphat_{\text{LS}}$ satisfies $\out = \reg \linphat_{\text{LS}}$.

Hence, we can minimize \eqref{eq:spice} by setting $\covnoise^* = \0$ and $\covp^* = \linphat_{\text{LS}} \linphat_{\text{LS}}^\top$, since this gives $\covout =  \out \out^\top$. 
Using Theorem~\ref{thm:original_sol} we thus get
\[
	\linphat(\covp^*, \covnoise^*) = \covp^* \reg^\top \covout^\dagger \out = \linphat_{\text{LS}} \out^\top (\out\out^\top)^\dagger \out = \linphat_{\text{LS}}.
\]

\subsection{Proof of Theorem \ref{thm:f}}\label{app:f}
Consider $(\covp, \covnoise) \in \covset$ such that $\out \in \range(\covout)$. From Theorem~\ref{thm:original_sol} it follows that
\[
	F(\covp, \covnoise) = J(\linphat(\covp, \covnoise); \covp, \covnoise).
\]
Using \eqref{eq:res_in_V} we see that $\out - \reg \linphat = \covnoise \covout^\dagger \out$, so 
\begin{align*}
	F(\covp, \covnoise) &= \| \covnoise \covout^\dagger \out \|_{\covnoise^\dagger}^2 + \| \covp \reg^\top \covout^\dagger \out\|_{\covp^\dagger}^2 + \frac{1}{\|\out\|_2^2} \tr\{ \covout\} \\
			    &= \| \out \|_{\covout^\dagger}^2 + \frac{1}{\|\out\|_2^2} \tr\{ \covout \}.
\end{align*}
Also note that, since $\out \in \range(\covout)$, it follows that $\out = \covout\covout^\dagger \out = \covout^\dagger \covout \out$ and thus the criterion in \eqref{eq:spice} can be rewritten as 
\[
	\| \out \out^\top - \covout \|_{\covout^\dagger}^2 = \|\out \|_2^2 \| \out \|_{\covout^\dagger}^2 + \tr\{ \covout \} - 2 \| \out \|_2^2.
\]
The theorem follows by combining these two expressions.

\subsection{Proof of Theorem \ref{thm:main}}\label{app:main}
Let $\linp^* \in \linpset^*$. By definition of $\linpset^*$ there exist a pair $(\covp^*, \covnoise^*) \in \covset$ that minimize \eqref{eq:spice} such that $\linp^* = \linphat(\covp^*, \covnoise^*)$ defined in \eqref{eq:original_opt}. 

It follows from Theorem~\ref{thm:f} that
\begin{align*}
	\min_{\linp} G(\linp) &= \inf_{\substack{ (\covp, \covnoise) \in \covset \\ \linp \in \linpset(\covp, \covnoise)}} J(\linp; \covp, \covnoise) \\
	&= \inf_{(\covp, \covnoise) \in \covset} F(\covp, \covnoise) = F(\covp^*, \covnoise^*).
\end{align*}
Next note that
\begin{align*}
	G(\linp^*) &= \inf_{\substack{ (\covp, \covnoise) \in \covset \\ \text{s.t. } \linp^* \in \linpset(\covp, \covnoise)}} J(\linp^*; \covp, \covnoise) \leq J(\linp^*; \covp^*, \covnoise^*) \\
		   &= F(\covp^*, \covnoise^*) = \min_{\linp} G(\linp)
\end{align*}
This implies that
\[
	\min_{\linp} G(\linp) = G(\linp^*) = F(\covp^*, \covnoise^*),
\]
and thus $\linp^*$ is a minimizer of $G(\linp)$. Hence $\linpset^* \subseteq \argmin_{\linp} G(\linp)$.

For the other direction, consider $\linp_G \in \argmin_{\linp} G(\linp)$, and assume that the infimum in \eqref{eq:G} is attained for $\linp_G$. That is, we assume that there are $(\covp_G, \covnoise_G) \in \covset$ such that $\linp_G \in \linpset(\covp_G, \covnoise_G)$ and
\[
	G(\linp_G) = \min_{\linp} G(\linp) = J(\linp_G; \covp_G, \covnoise_G).
\]
Note that
\[
	J(\linp_G; \covp_G, \covnoise_G) \geq \min_{\linp \in \linpset(\covp_G, \covnoise_G)} J(\linp; \covp_G, \covnoise_G) = F(\covp_G, \covnoise_G).
\]
However, since $\min_{\linp} G(\linp) = F(\covp^*, \covnoise^*) \leq F(\covp_G, \covnoise_G)$, the above inequality must actually be an equality. That is, $(\covp_G, \covnoise_G)$ is a solution to \eqref{eq:spice} and 
\[
	J(\linp_G; \covp_G, \covnoise_G) = \min_{\linp \in \linpset(\covp_G, \covnoise_G)} J(\linp; \covp_G, \covnoise_G).
\]
Finally note that it follows by Theorem~\ref{thm:original_sol} that the right-hand side has a unique minimizer, so $\linp_G = \linphat(\covp_G, \covnoise_G) \in \linpset^*$.

\subsection{Proof of Lemma~\ref{lem:diagonal}}\label{app:diagonal}
With $\covgenp = \text{diag}(a_1, \ldots, a_d)$ we get
\begin{align*}
	f(\genp, \covgenp, \weight) &= \| \genp \|_{\covgenp^\dagger}^2 + \frac{1}{\| \out \|_2^2} \tr\{ \weight \covgenp\} \\
				    &= \sum_{i=1}^{d} \left( a_i^\dagger x_i^2 + \frac{1}{\| \out\|_2^2}  a_i w_{i,i} \right)
\end{align*}
where $w_{i,i}$ are the diagonal elements of $\weight$ and 
\[
	a_i^\dagger = \begin{cases} 1/a_i & \text{if } a_i \neq 0 \\ 0 & \text{if } a_i = 0 \end{cases}.
\]
Note that setting $a_i = 0$ will make the corresponding term in the sum equal to zero. However, to satisfy the constraint $\genp \in \range(\covgenp)$, we must have $a_i > 0$ if $x_i \neq 0$. So if $x_i \neq 0$, then $a_i^\dagger = 1/a_i$, and we can find the optimal $a_i$ by taking the derivative and setting it equal to zero. This gives
\[
	a_i = \frac{ \|\out\|_2 |x_i|}{ \sqrt{ w_{i,i}}}.
\]
We note that this formula also works for the case that $x_i = 0$. Inserting this back into the sum we get
\[
	h(\genp, \weight; \covgenpset) = \frac{2}{\|\out\|_2} \sum_{i=1}^d \sqrt{w_{i,i}} | x_i| = \frac{2}{\|\out\|_2} \| \sqrt{ \eye \odot \weight} \genp \|_1.
\]
\subsection{Proof of Lemma~\ref{lem:full}}\label{app:full}
If $\genp = \0$ then we get $h(\genp, \weight; \covgenpset) = 0$ by setting $\covgenp = \0$. For $\genp \neq \0$, we will compute a lower bound on $f(\genp, \covgenp, \weight)$ and show that this lower bound can be achieved unless $\weight \genp = \0$.

Consider any $\genp \neq \0$, and any positive semi-definite $\covgenp$ such that $\genp \in \range(\covgenp)$ with rank $m$. Since $\genp \neq \0$, we must have $m \geq 1$.
Factorize $\covgenp$ as
\[
\covgenp = \covpu \covpe \covpu^\top
\]
where $\covpu \in \real^{d \times m}$ has orthonormal columns and $\covpe = \diag(\lambda_1, \ldots, \lambda_m) \succ \0$. Let $\covpuc_i$ be the $i$th column of $\covpu$. Then
$\covgenp^\dagger = \covpu \covpe^{-1} \covpu^\top$ and
\begin{align*}
	f(\genp, \covgenp; \weight) &= \tr( \covpu^\top \genp\genp^\top \covpu \covpe^{-1}) + \frac{1}{\|\out\|_2^2} \tr\left\{ \covpu^\top \weight \covpu \covpe \right\} \\
				    &= \sum_{i=1}^m \left(  \frac{1}{\lambda_i} | \covpuc_i^\top \genp|^2 + \frac{\lambda_i}{\| \out \|_2^2} \covpuc_i^\top \weight \covpuc_i \right)
\end{align*}
If we let $a_i = | \covpuc_i^\top \genp |/ \sqrt{\lambda}$ and $b_i = \frac{\sqrt{\lambda}}{\| \out\|_2} \sqrt{\covpuc_i^\top \weight \covpuc_i}$, then each term in the sum can be written as
\[
	a_i^2 + b_i^2 \geq 2 ab = \frac{2}{\|\out\|_2} | \covpuc_i^\top \genp | \sqrt{ \covpuc_i^\top \weight \covpuc_i}.
\]
Hence, 
\begin{align}
	f(\genp, \covgenp; \weight) &\geq \frac{2}{\|\out\|_2} \sum_{i=1}^m |\covpuc_i^\top \genp| \sqrt{ \covpuc_i^\top \weight \covpuc_i} \label{eq:ineq1} \\
				    &= \frac{2}{\|\out\|_2} \sum_{i=1}^m \| \covpuc_i \covpuc_i^\top \genp \|_{\weight} \nonumber \\
				    &\geq \frac{2}{\|\out\|_2} \left\| \sum_{i=1}^m \covpuc_i \covpuc_i^\top \genp \right\|_{\weight} = \frac{2}{\|\out\|_2} \| \genp \|_\weight \label{eq:ineq2}
\end{align}
where the second inequality follows from the triangle inequality, and the last equality follows from 
\[
	\sum_{i=1}^m \covpuc_i \covpuc_i^\top \genp = \covpu \covpu^\top \genp = \covpu \covpu^\dagger \genp = \genp
\]
since $\genp \in \range(\covgenp) = \range(\covpu)$.

Assuming that $\weight \genp \neq \0$, we can achieve this lower bound by using
\[
	\what{\covgenp} = \frac{\|\out\|_2}{\| \genp\|_\weight} \genp \genp^\top,
\]
which clearly satisfy $\genp \in \range(\what{\covgenp})$. To see this, note that for $\genp \neq \0$ and $\alpha \neq 0$
\[
 (\alpha \genp \genp^\top)^\dagger = \frac{ \genp \genp^\top}{ \alpha \| \genp\|_2^4}.
\]
If $\weight \genp = \0$ then the lower bound just states $f(\genp, \covgenp, \weight) \geq 0$. For $\genp = \0$ this lower bound can be achieved by setting $\covgenp = \0$. However, if $\genp \neq \0$, then the constraint $\genp \in \range(\covgenp)$ ensures that 
\[
	\| \genp \|_{\covgenp^\dagger}^2 > 0.
\]
But by choosing $\what{\covgenp}(\lambda) = \lambda \genp \genp^\top$, we get 
\[
	f(\genp, \what{\covgenp}(\lambda), \weight) = \frac{1}{\lambda} \rightarrow 0
\]
as $\lambda \rightarrow \infty$. Hence the lower bound can be reached in the limit also for these $\genp$, so for all $\genp$, 
\[
	h(\genp; \weight; \covgenpset) = \frac{2}{\|\out\|_2} \|\genp\|_{\weight}.
\]
But in the special case that $\weight \genp = \0$ but $\genp \neq \0$, the infimum in \eqref{eq:G} cannot be attained.

\bibliographystyle{ieeetr}
\bibliography{refs_spicenote}

\end{document}